\documentclass{amsart}

\usepackage{amssymb,amsmath,amsthm,latexsym,enumerate}
  \theoremstyle{definition}
  \newtheorem{definition}{Definition}[section]
  \newtheorem{notation}[definition]{Notation}

  \newtheorem{remark}[definition]{Remark}

  \theoremstyle{plain}
  \newtheorem{lemma}[definition]{Lemma}
  
  \newtheorem{theorem}[definition]{Theorem}
  \newtheorem{corollary}[definition]{Corollary}

\newcommand{\complex}{\mathbb{C}}
  
\newcommand{\field}{\mathbb{F}}

\begin{document}

\title[]
{On the irreducible representations of the Jordan triple system of $p \times q$ matrices}

\author{Hader A. Elgendy}

\address{Department of Mathematics, Faculty of Science, Damietta
University, Damietta 34517, Egypt}

\email{haderelgendy42@hotmail.com}

\begin{abstract}
Let $\mathcal{J}_{\field}$ be the Jordan triple system of all $p \times q$ ($p\neq q$; $p,q >1)$ rectangular matrices over a field $\field$ of characteristic 0 with the triple product $\{x,y,z\}= x y^t z+ z y^t x $, where $y^t$ is the transpose of $y$.  We study the universal associative envelope  $\mathcal{U}(\mathcal{J}_{\field})$ of $\mathcal{J}_{\field}$ and show that $\mathcal{U}(\mathcal{J}_{\field}) \cong  M_{p+q \times p+q}(\field)$, where $M_{p+q\times p+q} (\field)$ is the ordinary associative algebra of all $(p+q) \times (p+q)$ matrices over $\field$.  It follows that there exist only one nontrivial irreducible representation of  $\mathcal{J}_{\field}$. The center of $\mathcal{U}(\mathcal{J}_{\field})$ is deduced.
 \end{abstract}
\date{\textit{\today}}
\maketitle
\section{Introduction}
A vector space $V$ over a field $\field$ of characteristic 0 equipped with a triple product $\{a, b , c\}$ is called a \emph{Jordan triple system} if
\[\{x,y,z\}=\{z,y,x\},\]
\[\{u, v,\{x, y, z\}\} = \{\{u, v, x\}, y, z\} - \{x,\{v, u, y\}, z\} + \{x, y,\{u, v, z\}\},\]
for all $x, y, z, u, v \in V$.

Jordan structures appeared in many areas of mathematics like Lie Theory,  differential geometry  and  analysis  \cite{Koecher1,  2, Kaup, Ko1, Koecher, Upmeier}.  In addition to that Jordan triple systems have been used to find several solutions of the Yang-Baxter equation \cite{OK}. The linkages between Jordan structures, Lie algebras, and projective geometries are given in \cite{6}. Jordan structures play also an important role in theoretical physics. They are appeared in the theory of superstrings \cite{Corrigan9,Fairlie4, Foot8, Goddard5, Gursey7, Sierra6}, and in the theory of colour and confinement \cite{Gunaydm1}, in supersymmetry \cite{Gunaydin1}.
The description of some of these applications has been given in the survey \cite{37}. More information about Jordan triple systems can be found in
  \cite{Mb, Ne}.
It is well known that to every associative algebra $A$ one can relate a Jordan triple system $J$ with the triple product $\{x,y,z\} = xyz + zyx$. A Jordan triple system is \emph{special} if it can be imbedded as a subtriple of some $J$, otherwise it is \emph{exceptional}.  A \emph{representation} of a Jordan triple system $J$ is a Jordan triple homomorphism $\Theta: J \to (End V )_{-}$, where $End V $ is the space of endomorphisms of a vector space $V$ to itself.  
   A representation $\rho$ is called \emph{irreducible} if the only invariant subspaces of $V$ under $\rho$ are the trivial ones, $\{0\}$ and $V$.  It is known that any Jordan algebra gives rise to a Jordan triple system. One of the most important examples of a Jordan triple system which doesn't come from a bilinear product is the rectangular matrices $M_{p\times q}(\field)$ with the triple product $xy^t z+ zy^t x$; if $p\neq q$ there is no natural way to multiply two $p\times q$ matrices to get a third $p\times q$ matrix, see \cite{Mc}. This example shows the necessity of a ternary product. This example is a special Jordan triple system  (see the map $\Theta$ of Corollary \ref{xo} (of the present paper)).

The problem of the classification of the representations of a special Jordan triple system can be converted into a problem of an associative algebra by passage to the universal associative envelope of the Jordan triple system. In \cite{2}, it was shown that the universal (associative) envelope of any Jordan triple system of finite dimension is finite-dimensional.  In \cite{elgendy4}, we showed that the universal associative envelope of the Jordan triple system $J$ all $n$ by $n$ $(n\geq 2)$ matrices over a field $\field$ of characteristic 0 (with respect to the product xyz+zyx), then the universal (associative) envelope of $J$ is isomorphic to $M_{n\times n}(\field) \oplus M_{n\times n}(\field) \oplus M_{n\times n}(\field) \oplus M_{n\times n}(\field)$, where $M_{n\times n}(\field)$ is the ordinary associative algebra of all $n$ by $n$ matrices over $\field$. It follows that there are four nontrivial finite-dimensional irreducible representations of $J$. In \cite{elgendy5}, we have studied the representations of two special Jordan triple systems (with respect to the product $xyz+zyx$): The Jordan triple system $\mathcal{J_S}$ of all symmetric $n$ by $n$ $(n\geq 2)$ matrices over a field $\field$ of characteristic zero, and the Jordan triple system $\mathcal{J_H}$ of all Hermitian $n$ by $n$ $(n\geq 2)$ matrices over the complex numbers $\complex$. We proved that the universal (associative) envelope of $\mathcal{J_S}$ is isomorphic to $M_{n\times n}(\field) \oplus M_{n\times n}(\field)$, while the universal (associative) envelope of $\mathcal{J_H}$ is isomorphic to $M_{n\times n}(\complex) \oplus M_{n\times n}(\complex)\oplus M_{n\times n}(\complex) \oplus M_{n\times n}(\complex)$. We deduced that the Jordan triple system $\mathcal{J_S}$ has two nontrivial finite-dimensional inequivalent irreducible representations, while the Jordan triple system $\mathcal{J_H}$ has four nontrivial inequivalent finite-dimensional irreducible representations.

In the present paper, we study the universal associative envelope of the special Jordan triple system $\mathcal{J}_{\field}$ of all $p\times q$ $(p\neq q;\, p,q > 1)$ rectangular matrices with the triple product $xy^t z+ zy^t x$.
This paper is organized as follows. In Section \ref{u2}, we construct the universal (associative) envelope $\mathcal{U}(\mathcal{J}_{\field})$ of $\mathcal{J}_{\field}$ and derive some identities of $\mathcal{U}(\mathcal{J}_{\field})$. In Section \ref{u3}, we prove that
 $\mathcal{U}(\mathcal{J}_{\field})$ is isomorphic to $M_{p + q\times p + q}(\field)$, where $M_{p + q\times p + q}(\field)$ is the ordinary associative algebra of all $(p + q)$ by $(p + q)$ matrices over $\field$ (Theorem \ref{S}). We also deduce that $\mathcal{J}_{\field}$ has only one nontrivial irreducible representation and determine the explicit form of this representation (Corollary \ref{xo}). The center of $\mathcal{U}(\mathcal{J}_{\field})$ is also determined (Lemma \ref{center}).

\section{The universal associative envelope of the special Jordan triple system of rectangular matrices}\label{u2}
\begin{definition}\label{def}
Let $\mathcal{J}_{\field}$ be the Jordan triple system of the rectangular matrices $M_{p\times q}(\field)$ ($p,q>1;\,p\neq q$) over a field $\field$ of characteristic 0 with the triple product
\[\{x,y,z\}= xy^t z+ zy^t x,\]
where $y^t$ is the transpose of $y$.
\end{definition}
\begin{definition}
We let $\Omega_1 = \{1, \dots, p\}$, $\Omega_2 = \{1, \dots, q\}$, and $\Omega_3 = \{p+1, \dots, p+q\}$ be three finite index sets.
Let $\mathfrak{B} = \{E_{i,j}\mid  i\in \Omega_1;\, j \in \Omega_2\}$ be a basis of $\mathcal{J}_{\field}$, where $E_{i,j}$  denotes the $p$-by-$q$ matrix with a single 1, in the $i$th row and $j$th column, and zeros elsewhere.
\end{definition}
 \begin{notation}
 Throughout this paper, we use the following notations:
\begin{itemize}
   \item $\delta_{i,j}$ for the Kronecker delta, and $\widehat{\delta}_{i,j}= 1 - \delta_{i,j}$.
  \item $\Delta_{i,L} = 1\,\,  \text{if} \,\, i \in L,\,\, \text{and} \,\, 0 \,\,\text{otherwise}$.
  \end{itemize}
    \end{notation}
 Let $\mathbb{X} =\{G_{i,j}\mid i\in \Omega_1;\,  j\in \Omega_2\}$ be a set of symbols in bijection with $\mathfrak{B}$ and let $\Phi: \mathfrak{B} \to \mathbb{X}$ realize the bijection $(
 \Phi({E}_{i,j})= G_{i,j})$.  Let $\mathfrak{F}$ be the free associative algebra generated by $\mathbb{X}$. We extend $\Phi$ to a map  $\Phi:\mathcal{J}_{\field}\to \mathfrak{F}$ (by linearity).
Let $I$ be the two-sided ideal of $\mathfrak{F}$ generated by all the elements of the form:
 \[  G_{i,j} G_{k,\ell} G_{s,t} + G_{s,t} G_{k,\ell} G_{i,j}- \Phi(\{E_{i,j}, E_{k,\ell}, E_{s,t}\})\quad (i,k,s\in \Omega_1,\,\, j,\ell,t\in \Omega_2).\]
  The universal associative envelope $\mathcal{U}(\mathcal{J}_{\field})$ of $\mathcal{J}_{\field}$ is the quotient $\mathfrak{F}/I$. Let $\pi: \mathfrak{F} \to \mathcal{U}(\mathcal{J}_{\field})$ be the projection, then the map $\iota =  \pi \circ \Phi$ maps $\mathcal{J}_{\field}$ to $\mathcal{U}(\mathcal{J}_{\field})$.
\subsection{Identities of the universal associative envelope}\label{identity}
In this section we get identities of $\mathcal{U}(\mathcal{J}_{\field})$ that we use in the proof of the main results of the next section.
\begin{lemma}\label{R}
In $\mathcal{U}(\mathcal{J}_{\field})$, the following  identities hold:
\begin{enumerate}[$(I)$]
\item $G_{i,j} G_{i,j} G_{i,j}\equiv G_{i,j}\quad (i\in \Omega_1,\,\, j\in \Omega_2)$,  \label{f1}
\item $G_{i,j} G_{k,\ell} \equiv 0 \quad (i\neq k,\,\, j\neq \ell;\,\, i,k \in \Omega_1, \,\, j,\ell \in \Omega_2)$, \label{f2}
\item $G_{i,j}G_{i,\ell} \equiv G_{1,j}G_{1,\ell} \quad (j\neq \ell;\,\, i \in \Omega_1\setminus \{1\},\,\, j,\ell \in \Omega_2)$, \label{f3}
\item $G_{i,j} G_{k,j}\equiv  G_{i,1}G_{k,1} \quad (i\neq k;\,\, i,k \in \Omega_1, \,\, j \in \Omega_2\setminus\{1\})$,  \label{f4}
\item $ G_{1,j} G_{1,\ell} G_{1,s} \equiv 0\quad (j\neq \ell \neq s;\,\, j,\ell, s \in \Omega_2)$, \label{f5}
\item $G_{i,1} G_{k,1} G_{t,1} \equiv 0\quad (i\neq k \neq t;\,\, i,k,t \in \Omega_1)$, \label{f6}
\item $G_{1,1}G_{i,1} G_{i,1} \equiv G_{1,j}G_{1,j} G_{1,1} \quad (i \in \Omega_1\setminus\{1\}, \,\, j\in \Omega_2\setminus\{1\})$, \label{f7}
\item $G_{i,1} G_{i,1} G_{1,1} \equiv G_{1,1} G_{1,j} G_{1,j}\quad (i \in \Omega_1\setminus\{1\},\,\, j\in \Omega_2\setminus\{1\} )$.\label{f8}
 \end{enumerate}
\end{lemma}
\begin{proof}
For \eqref{f1}: It is obvious, since $2G_{i,j}G_{i,j}G_{i,j} \equiv 2G_{i,j} $  ($i\in \Omega_1$;\, $j\in \Omega_2$). For \eqref{f2}: Let $i, k \in \Omega_1$, $j,\ell \in \Omega_2$, $i\neq k$, and $j\neq \ell$. By \eqref{f1} (of the present lemma), we have  $G_{i,j} G_{i,j} G_{i,j}\equiv G_{i,j}$. Multiplying by $G_{k,\ell}$ from the right, we get
\begin{equation}\label{xy}
G_{i,j} G_{i,j} G_{i,j} G_{k,\ell} \equiv G_{i,j} G_{k,\ell}. 
\end{equation} 
Using $G_{i,j} G_{i,j} G_{k,\ell} \equiv -G_{k,\ell}  G_{i,j} G_{i,j} $ in \eqref{xy} gives
\begin{equation*}
-G_{i,j}  G_{k,\ell}  G_{i,j} G_{i,j} \equiv G_{i,j} G_{k,\ell},
\end{equation*} 
which implies  \eqref{f2}, since $G_{i,j}  G_{k,\ell}  G_{i,j}\equiv 0$.
 For \eqref{f3}: Let $i \in \Omega_1 \setminus\{1\}$, $j,\ell \in \Omega_2$, and $j\neq \ell$, we have
\begin{equation*}
G_{i,j} G_{1,j} G_{1,j} + G_{1,j} G_{1,j} G_{i,j}- G_{i,j}\equiv 0.
\end{equation*}
Multiplying by $G_{i,\ell}$ from the right and observing that $G_{1,j} G_{i,\ell} \equiv 0$ (by \eqref{f2} (of the present lemma)), we get \begin{equation}\label{q1}
 G_{1,j} G_{1,j} G_{i,j} G_{i,\ell}  - G_{i,j}G_{i,\ell}  \equiv 0.
\end{equation}
Using $G_{1,j} G_{i,j} G_{i,\ell} \equiv -G_{i,\ell} G_{i,j} G_{1,j} + G_{1,\ell}$ in \eqref{q1} gives
\begin{equation*}
 G_{1,j} \left(-G_{i,\ell} G_{i,j} G_{1,j} + G_{1,\ell}\right)  - G_{i,j}G_{i,\ell}  \equiv 0,
\end{equation*}
which implies \eqref{f3}; since $G_{1,j} G_{i,\ell} \equiv 0$ (by \eqref{f2} (of the present lemma)). For \eqref{f4}: Let $i,k \in \Omega_1$, $ j \in \Omega_2\setminus \{1\}$, and  $i\neq k$. We have
\begin{equation*}
G_{i,j} G_{i,1} G_{i,1} + G_{i,1} G_{i,1} G_{i,j} - G_{i,j}\equiv 0.
\end{equation*}
Multiplying by $G_{k,j}$ from the right and observing that $G_{i,1} G_{k,j} \equiv 0$ (by \eqref{f2} (of the present lemma)), we obtain
\begin{equation}\label{q2}
G_{i,1} G_{i,1} G_{i,j}G_{k,j} - G_{i,j} G_{k,j}\equiv 0.
\end{equation}
Using $G_{i,1} G_{i,j} G_{k,j} \equiv - G_{k,j} G_{i,j} G_{i,1} + G_{k,1}$ in \eqref{q2} gives
\begin{equation*}
G_{i,1} (- G_{k,j} G_{i,j} G_{i,1} + G_{k,1}) - G_{i,j} G_{k,j}\equiv 0,
\end{equation*}
which implies \eqref{f4}; since $G_{i,1} G_{k,j} \equiv 0$ (by \eqref{f2} (of the present lemma)).
For \eqref{f5}: Let $i\in \Omega_1 \setminus \{1\}$, $j, \ell, s \in \Omega_2$, and $j\neq \ell \neq s$.  By \eqref{f3} (of the present lemma), we have
\begin{align*}
G_{i,j} G_{i,\ell} \equiv G_{1,j} G_{1,\ell}.
\end{align*}
Multiplying by $G_{1,s}$ from the right, we get
 \begin{align*}
G_{i,j} G_{i,\ell}G_{1,s} \equiv G_{1,j} G_{1,\ell}G_{1,s},
\end{align*}
which implies \eqref{f5}; since $G_{i, \ell}G_{1,s}\equiv 0$ (by \eqref{f2} (of the present lemma)).
For \eqref{f6}:  Let $i,k,t \in \Omega_1$, $i \neq k\neq t$, and $j\in \Omega_2\setminus \{1\} $. By \eqref{f3} (of the present lemma), we have
\begin{align*}
G_{i,j} G_{k,j}  \equiv  G_{i,1} G_{k,1}.
\end{align*}
Multiplying by $G_{t,1}$ from the right gives
\begin{align*}
G_{i,j} G_{k,j} G_{t,1} \equiv  G_{i,1} G_{k,1} G_{t,1},
\end{align*}
which implies \eqref{f6}; since $G_{k,j}G_{t,1} \equiv 0$ (by \eqref{f2} (of the present lemma)).
For \eqref{f7}: Let $i\in \Omega_1\setminus \{1\}$ and $j \in \Omega_2\setminus \{1\}$.  By \eqref{f4} (of the present lemma), we have
\begin{align*}
G_{1,j} G_{i,j} \equiv G_{1,1} G_{i,1}.
\end{align*}
Multiplying by $G_{i,1}$ from the right and observing that $G_{i,j} G_{i,1}\equiv G_{1,j} G_{1,1}$ (by \eqref{f3} (of the present lemma)), we get  \eqref{f7}.
For \eqref{f8},  let $i \in \Omega_1\setminus \{1\}$, and $j\in \Omega_2\setminus \{1\}$,  we have
 \begin{align}\label{q3}
 G_{i,1} G_{i,1} G_{1,1}& \equiv - G_{1,1} G_{i,1} G_{i,1}+ G_{1,1}.
 \end{align}
 By \eqref{f7} (of the present lemma), we have $G_{1,1} G_{i,1} G_{i,1}\equiv  G_{1,j} G_{1, j} G_{1,1}$. Using this in \eqref{q3} gives
  \begin{align*}
 G_{i,1} G_{i,1} G_{1,1}&\equiv - G_{1,j} G_{1,j} G_{1,1}+ G_{1,1}
 \equiv G_{1,1} G_{1,j} G_{1,j}.
 \end{align*}
 This completes the proof.
\end{proof}
\begin{remark}\label{r}
By \eqref{f7} and \eqref{f8} of Lemma \ref{R}, for all $i,k \in \Omega_1\setminus\{1\}$, $i\neq k$, and $j\in \Omega_2\setminus\{1\}$,  we have
\[  G_{1,1}G_{i,1} G_{i,1} \equiv G_{1,j}G_{1,j} G_{1,1}\equiv G_{1,1}G_{k,1} G_{k,1},\]
    and
 \[  G_{i,1} G_{i,1} G_{1,1} \equiv G_{1,1} G_{1,j} G_{1,j}\equiv G_{k,1} G_{k,1} G_{1,1}.\]
  That is, the products  $G_{1,1}G_{i,1} G_{i,1}$ and  $G_{i,1} G_{i,1} G_{1,1}$ (resp. $G_{1,j}G_{1,j} G_{1,1}$ and $G_{1,1} G_{1,j} G_{1,j}$) do not depend on the choice of $i \in \Omega_1\setminus\{1\}$ (resp. $j\in \Omega_2\setminus\{1\}$).
\end{remark}
\begin{corollary}\label{R2}
In $\mathcal{U}(\mathcal{J}_{\field})$, the following  identities hold:
\begin{enumerate}[$(I)$]
\item $G_{1,\ell} G_{1,1} G_{1,1} G_{1,j} \equiv - \delta_{\ell, j} G_{1,1} G_{1,1} G_{1,j}G_{1,j} + G_{1,\ell} G_{1,j}\quad  (\ell, j \in \Omega_2 \setminus \{1\})$,\label{t1}
\item $G_{i,1}G_{1,1} G_{1,1} G_{k,1}\equiv - \delta_{i,k} G_{1,1} G_{1,1} G_{i,1} G_{i,1} + G_{i,1} G_{k,1} \quad  (i, k \in \Omega_1 \setminus \{1\})$,\label{t2}
\item $G_{1, j} G_{1,\ell} G_{1,\ell} G_{1,1} \equiv -\delta_{j,1} G_{1,1} G_{1,1} G_{1,\ell} G_{1,\ell} +  G_{1,j} G_{1,1}\quad  (j, \ell \in \Omega_2; \ell\neq 1) $,\label{t3}
\item $G_{1,1}G_{1,1} G_{1,j} G_{1,j} G_{1,1} \equiv G_{1,j} G_{1,j} G_{1, 1}\quad  (j \in \Omega_2 \setminus\{1\}) $,\label{t4}
 \item $G_{1, j} G_{1, j} G_{1,1} G_{1,1} G_{1,\ell} \equiv  G_{1,1} G_{1,1} G_{1,\ell}\quad (j,\ell \in \Omega_2\setminus\{1\})$,  \label{t77}
\item $G_{1,1}G_{1,1} G_{1,j} G_{1,j} G_{1,1}G_{i,1} \equiv G_{1, 1} G_{i,1} \quad  ( i \in \Omega_1 \setminus\{1\} ,\,\, j \in \Omega_2 \setminus \{1\})$, \label{t5}
\item $G_{1,1} G_{1,j} G_{1,j} G_{1,\ell} \equiv G_{1,1} G_{1,\ell}\quad (j,\, \ell \in \Omega_2 \setminus \{1\})$,\label{t6}
\item $G_{i,1}G_{1,1} G_{1,j} G_{1,j} \equiv G_{i,1} G_{1,1}\quad  ( i \in \Omega_1 \setminus\{1\} ,\,\, j \in \Omega_2 \setminus \{1\})  $.\label{t7}
\end{enumerate}
\end{corollary}
\begin{proof}
For $\eqref{t1}$: Let $\ell, j \in \Omega_2 \setminus \{1\}$, we have
\begin{align*}
G_{1,\ell} G_{1,1} G_{1,1} G_{1,j} &\equiv \left(- G_{1,1} G_{1,1} G_{1,\ell} + G_{1,\ell}\right)G_{1,j}
\\
&\equiv  - \delta_{\ell, j}G_{1,1} G_{1,1} G_{1,j} G_{1,j} +  G_{1,\ell}G_{1,j},
 \end{align*}
using $G_{1,\ell} G_{1,1} G_{1,1} \equiv - G_{1,1} G_{1,1} G_{1,\ell} + G_{1,\ell}$ and \eqref{f5} of Lemma \ref{R}. The proof of \eqref{t2} is similar. For \eqref{t3}: Let $\ell, j \in \Omega_2$ and $\ell\neq 1$, we have
\begin{align*}
G_{1, j} G_{1,\ell} G_{1,\ell} G_{1,1} &\equiv  G_{1,j}( -G_{1,1}G_{1,\ell} G_{1,\ell}+ G_{1,1})
\\
&\equiv  -\delta_{j,1} G_{11} G_{1,1} G_{1,\ell} G_{1,\ell}+  G_{1,j} G_{1,1},
\end{align*}
using $G_{1,\ell} G_{1,\ell} G_{1,1} \equiv  -G_{1,1}G_{1,\ell} G_{1,\ell}+ G_{1,1}$ and \eqref{f5} of Lemma \ref{R}.
For \eqref{t4}: Let $j\in \Omega_2\setminus \{1\}$ and choose any $t \in \Omega_1\setminus \{1\}$, we have
\begin{align*}
G_{1,1}G_{1,1} G_{1,j} G_{1,j} G_{1,1}& \equiv G_{1,1} G_{1,1}G_{1,1}  G_{t,1} G_{t, 1}
\equiv  G_{1, 1} G_{t,1}G_{t,1} \equiv G_{1,j} G_{1,j} G_{1,1},
\end{align*}
by \eqref{f7} and \eqref{f1} of Lemma \ref{R}. For \eqref{t77}: Let $j,\ell \in \Omega_2\setminus\{1\}$ and choose any $t\in \Omega_1\setminus\{1\}$, we have
\begin{align*}
 G_{1, j} G_{1, j} G_{1,1} G_{1,1} G_{1,\ell} \!\equiv\!  G_{1,1} G_{t, 1} G_{t, 1} G_{1,1} G_{1,\ell}
 \!\equiv\! G_{1,1} G_{1,1} G_{1,\ell} G_{1,\ell}G_{1,\ell}
\!\equiv \! G_{1,1} G_{1,1} G_{1,\ell},
\end{align*}
by \eqref{f7}, \eqref{f8}, and \eqref{f1} of Lemma \ref{R}. For \eqref{t5}:  Let $i \in \Omega_1 \setminus \{1\}$ and $j \in \Omega_2 \setminus \{1\}$, we have
\begin{align*}
G_{1,1}G_{1,1} G_{1,j} G_{1,j} G_{1,1}G_{i,1}  &\equiv G_{1,j} G_{1,j} G_{1,1} G_{i,1}
\equiv G_{1,1} G_{i,1} G_{i,1} G_{i,1}\equiv G_{1,1} G_{i,1},
\end{align*}
by \eqref{t4} (of the present lemma) and \eqref{f7}, \eqref{f1} of Lemma \ref{R}. For \eqref{t6}: Let $\ell,\, j \in \Omega_2 \setminus \{1\}$, we have
\begin{align*}
G_{1,1} G_{1,j} G_{1,j} G_{1,\ell} &\equiv G_{1,1} G_{1,\ell} G_{1,\ell} G_{1,\ell}
 \equiv  G_{1,1} G_{1,\ell},
\end{align*}
by Remark \ref{r} and \eqref{f1} of Lemma \ref{R}. For \eqref{t7}: Let $i \in \Omega_1 \setminus \{1\}$ and $j \in \Omega_2 \setminus \{1\}$, we have
\begin{align*}
G_{i,1}G_{1,1} G_{1,j} G_{1,j} \equiv G_{i,1} G_{i,1} G_{i,1} G_{1,1}\equiv G_{i,1} G_{1,1},
\end{align*}
by \eqref{f8} and \eqref{f1} of Lemma \ref{R}.  This completes the proof.
\end{proof}

\section{Main Results}\label{u3}
In this  section  we  present  the  main  results  of  this paper  on  the representations of the Jordan triple system $\mathcal{J}_{\field}$.
Our goal is to use the specialty of the Jordan triple system $\mathcal{J}_{\field}$  and the identities of Section \ref{identity} to get the decomposition  of $\mathcal{U}(\mathcal{J}_{\field})$ into matrix algebras.

\begin{theorem}\label{S}
With notation as above. If $\mathcal{J}_{\field}$ is the Jordan triple system of all $p\times q$ $(p\neq q;\, p, q> 1)$ rectangular matrices over a field $\field$ of characteristic 0 and $\mathcal{U}(\mathcal{J}_{\field})$ is the universal associative envelope of $\mathcal{J}_{\field}$, then
\[ \mathcal{U}(\mathcal{J}_{\field}) \cong  M_{p+q \times p+q}(\field),\]
where $M_{p+q \times p+q}(\field)$ is the ordinary associative algebra of all $(p+q) \times (p+q)$ matrices over $\field$.
\end{theorem}
\begin{proof}
 We set
  \begin{align*}
 & {A}_{i, k} = G_{i,1} G_{k,1} \quad (i, k \in \Omega_1;\,\,  i\neq k).
  \\
&{A}_{1, p+1} = G_{1,j} G_{1,j} G_{1,1} \quad (\text{for any}\,\, j\in \Omega_2\setminus \{1\}).
 \\
& {A}_{i, k} = G_{i,1} G_{1,1} G_{1,k-p}\quad (i\in  \Omega_1,\,k \in \Omega_3;\,\, (i,k)\neq (1,p+1)).
\\
&{A}_{i, k} = - {A}_{k,i} + G_{k,i-p} \quad (i\in \Omega_3,\,\, k\in  \Omega_1).
\\
&{A}_{i, k} = G_{1,i-p} G_{1,k-p} \quad (i,k\in  \Omega_3;\,\, i\neq k).
\\
&A_{1,1} = G_{1,1} G_{1,1} G_{1,j} G_{1,j} \quad (\text{for any}\,\, j\in \Omega_2\setminus \{1\}).
\\
& A_{i,i} = - G_{1,1} G_{1,1} G_{t,1} G_{t,1}+ G_{i,1} G_{i,1}\quad (i \in \Omega_1\setminus \{1\};\,\text{for any}\,\, t\in \Omega_1\setminus \{1\} ).
\\
&A_{i,i}  = - G_{1,1} G_{1,1} G_{1,j} G_{1,j}+ G_{1,i-p} G_{1,i-p} \quad (i\in \Omega_3;\,\, \text{for any}\,\, j\in \Omega_2\setminus \{1\}).
\end{align*}
We wish to show that the elements $A_{i,j}$ (for all  $i,j \in \Omega_1 \cup \Omega_3$) satisfy the multiplication table for matrix units.  We first observe that the elements $A_{1,1}$, ${A}_{1, p+1}$, and the first term of $A_{i, i}$ $(i\in (\Omega_1 \cup \Omega_3)\setminus\{1\})$ do not depend on the choice of $j\neq 1$ (see Remark \ref{r}).  Let $i \in (\Omega_1 \cup \Omega_3)\setminus \{1\}$ and choose any $j \in \Omega_2 \setminus \{1\}$, we have
\begin{align*}
 A_{1,i} &= \Delta_{i,\Omega_1} G_{1,1} G_{i,1} + \delta_{i,p+1} G_{1,j} G_{1,j} G_{1,1} +\widehat{\delta}_{i,p+1} \Delta_{i,\Omega_3} G_{1,1} G_{1,1} G_{1,i-p}.
 \\
A_{i,1} &= \Delta_{i,\Omega_1} G_{i,1} G_{1,1}+ \Delta_{i,\Omega_3}(- {A}_{1,i} + G_{1,i-p}).
\end{align*}
For all $i \in (\Omega_1 \cup \Omega_3)\setminus \{1\}$, we first consider the following four products: $A_{1,i} A_{1,1}$,  $A_{1, 1} A_{1,i}$,  $A_{1,1} A_{i,1}$, and $A_{i,1} A_{1,1}$.
For $A_{1,i} A_{1,1}$: We have
\begin{align}\label{p1}
A_{1,i} A_{1,1} &=  \Delta_{i,\Omega_1} G_{1,1} G_{i,1}G_{1,1}G_{1,1} G_{1,j} G_{1,j} \\&\quad+ \delta_{i,p+1} G_{1,j} G_{1,j} G_{1,1}G_{1,1}G_{1,1} G_{1,j} G_{1,j}\notag \\&\quad +\widehat{\delta}_{i,p+1} \Delta_{i,\Omega_3}  G_{1,1} G_{1,1} G_{1,i-p}G_{1,1}G_{1,1} G_{1,j} G_{1,j}\notag
\\&
\equiv  0,\notag
\end{align}
since $G_{11} G_{11} G_{11} \equiv G_{11}$ (by \eqref{f1} of Lemma \ref{R}) and the elements, $G_{1,1} G_{i,1}G_{1,1}$ ($i\neq 1$),  $G_{1,j} G_{1,1} G_{1,j}$ ($j\neq 1)$ and $G_{1,1} G_{1,i-p}G_{1,1}$ ($i\neq p+1$) vanish (by \eqref{f6}, \eqref{f5} of Lemma \ref{R}),

For $ A_{1, 1} A_{1,i}$: We have
\begin{align}\label{z}
 A_{1, 1} A_{1,i} &= \Delta_{i,\Omega_1}   G_{1,1}G_{1,1} G_{1,j} G_{1,j} G_{1,1} G_{i,1} + \delta_{i,p+1} G_{1,1}G_{1,1} G_{1,j} G_{1,j} G_{1,j} G_{1,j}G_{1,1} \\
&\quad +\widehat{\delta}_{i,p+1} \Delta_{i,\Omega_3} G_{1,1}G_{1,1} G_{1,j} G_{1,j} G_{1,1} G_{1,1} G_{1,i-p}\notag
\\
&\equiv \Delta_{i,\Omega_1}  G_{1,1} G_{i,1}  + \delta_{i,p+1} G_{1,1}G_{1,1} G_{1,j} G_{1,j}G_{1,1}\notag
 \\
&\quad +\widehat{\delta}_{i,p+1} \Delta_{i,\Omega_3} G_{1,1}G_{1,1}  G_{1,1} G_{1,1} G_{1,i-p}\notag
\\
& \equiv \Delta_{i,\Omega_1}  G_{1,1} G_{i,1} + \delta_{i,p+1} G_{1,j} G_{1,j} G_{1,1}\notag
+\widehat{\delta}_{i,p+1} \Delta_{i,\Omega_3}  G_{1,1} G_{1,1} G_{1,i-p}
\\&
= A_{1,i},\notag
\end{align}
by \eqref{t5}, \eqref{t4}, \eqref{t77} of Corollary \ref{R2} and \eqref{f1} of Lemma \ref{R}.
For $A_{1,1} A_{i,1}$: We have
\begin{align*}
A_{1,1} A_{i,1}& = \Delta_{i,\Omega_1} A_{1,1} G_{i,1} G_{1,1}+ \Delta_{i,\Omega_3}(- A_{1,1} {A}_{1,i} + A_{1,1} G_{1,i-p}).
\end{align*}
We observe that $A_{1,1}G_{i,1} = G_{1,1}G_{1,1} G_{1,j} G_{1,j} G_{i,1}\equiv 0$; since  $G_{1,j} G_{i,1}\equiv 0$ (by \eqref{f2} of Lemma \ref{R}). Using this and \eqref{z} (of the present proof), \eqref{t4} and \eqref{t6} of Corollary \ref{R2}, we obtain
\begin{align*}
A_{1,1} A_{i,1}& =  \Delta_{i,\Omega_3}(- A_{1,i} + A_{1,1} G_{1,i-p})
\\
& =  \Delta_{i,\Omega_3}(- A_{1,i} + G_{1,1}G_{1,1} G_{1,j} G_{1,j} G_{1,i-p})
\\
& \equiv \Delta_{i,\Omega_3}(- A_{1,i} + \delta_{i,p+1} G_{1,j} G_{1,j} G_{1,1} + \widehat{\delta}_{i,p+1} G_{1,1}G_{1,1} G_{1,i-p})
\\
&\equiv 0.
\end{align*}
For $A_{i,1} A_{1,1}$: We have
\begin{align*}
A_{i,1} A_{1,1} & = \Delta_{i,\Omega_1} G_{i,1} G_{1,1} A_{1,1}  + \Delta_{i,\Omega_3}(- {A}_{1,i}A_{1,1}  +  G_{1,i-p}A_{1,1})
\\
& = \Delta_{i,\Omega_1} G_{i,1} G_{1,1}  G_{1,1} G_{1,1} G_{1,j} G_{1,j}  + \Delta_{i,\Omega_3} G_{1,i-p} G_{1,1} G_{1,1} G_{1,j} G_{1,j}
\\
&\equiv  \Delta_{i,\Omega_1} G_{i,1} G_{1,1} G_{1,j} G_{1,j}  + \Delta_{i,\Omega_3} \big[\delta_{i,p+1} G_{1,1} G_{1,1} G_{1,1} G_{1,j} G_{1,j}
\\&\quad+  \widehat{\delta}_{i,p+1} G_{1,i-p} G_{1,1} G_{1,1} G_{1,i-p} G_{1,i-p}\big]
\\
&\equiv \Delta_{i,\Omega_1} G_{i,1} G_{1,1}
 + \Delta_{i,\Omega_3} \big[\delta_{i,p+1} G_{1,1}G_{1,j} G_{1,j} \\&\quad + \widehat{\delta}_{i,p+1} (- G_{1,1} G_{1,1} G_{1,i-p}+ G_{1,i-p}) \big]
 \\
 & = A_{i,1},
\end{align*}
by \eqref{p1} (of the present proof),  \eqref{f1} of Lemma \ref{R}, Remark \ref{r}, and \eqref{t7}, \eqref{t1} of Corollary \ref{R2}.
Summarizing, for all $ i \in (\Omega_1\cup \Omega_3)\setminus \{1\}$, we have
\begin{align}\label{pr1}
 A_{1,i} A_{1,1} = 0 =  A_{1,1} A_{i,1},\,\,    A_{1,1} A_{1,i} =  A_{1,i},\,\,   A_{i,1} A_{1,1} = A_{i,1}.
\end{align}
Throughout the rest of the proof, we assume that $i,k \in (\Omega_1\cup \Omega_3)\setminus \{1\}$. Using the products of \eqref{pr1}, we get
\begin{equation}\label{ss1}
A_{1,i} A_{1,k} = A_{1,i} A_{1,1} A_{1,k} = 0,\quad  A_{i,1} A_{k,1} = A_{i,1} A_{1,1} A_{k,1} = 0.
\end{equation}
We next consider the two products: $A_{1,i} A_{k,1}$ and $A_{i,1} A_{1,k}$.

For $A_{1,i} A_{k,1}$: Using \eqref{ss1} (of the present proof), we get
\begin{align*}
A_{1,i} A_{k,1} &= \Delta_{k,\Omega_1} A_{1,i} G_{k,1} G_{1,1}+ \Delta_{k,\Omega_3} A_{1,i} G_{1,k-p}.
\\&= \Delta_{k,\Omega_1}\big( \Delta_{i,\Omega_1}  G_{1,1} G_{i,1}
+  \delta_{i,p+1} G_{1,j} G_{1,j} G_{1,1}
\\&\qquad +\widehat{\delta}_{i,p+1} \Delta_{i,\Omega_3} G_{1,1} G_{1,1} G_{1,i-p}\big)G_{k,1} G_{1,1}
\\&\qquad
+ \Delta_{k,\Omega_3} \big(\Delta_{i,\Omega_1} G_{1,1} G_{i,1}
 + \delta_{i,p+1}  G_{1,j} G_{1,j} G_{1,1} \\&\qquad +\widehat{\delta}_{i,p+1} \Delta_{i,\Omega_3}  G_{1,1} G_{1,1} G_{1,i-p} \big)G_{1,k-p} .
\end{align*}
By \eqref{f2}, \eqref{f5}, and \eqref{f6} of Lemma \ref{R}, the following products vanish: $G_{i,1}G_{k,1} G_{1,1}$ ($i\neq k$), $G_{1,1} G_{k,1} G_{1,1}$, $G_{1,i-p} G_{k,1}$ ($i\neq p+1$), $G_{i,1} G_{1,k-p}$ ($k\neq p+1$), $G_{1,j} G_{1,1}  G_{1,k-p}$ ($k\neq p+1$), $G_{1,1} G_{1,i-p}  G_{1, k-p}$($i\neq k, p+1$). It follows that
\begin{align*}
A_{1,i} A_{k,1} &= \Delta_{k,\Omega_1} \delta_{i,k} G_{1,1} G_{k,1}G_{k,1} G_{1,1}
+ \Delta_{k,\Omega_3} \big( \delta_{i,p+1}\delta_{k,p+1}  G_{1,j} G_{1,j} G_{1,1}  G_{1,1} \\&\quad + \Delta_{i,\Omega_3}\widehat{\delta}_{i,p+1}\delta_{i,k} G_{1,1} G_{1,1} G_{1, k-p}  G_{1, k-p}\big).\notag
\end{align*}
We now choose any  $\ell \in  \Omega_1\setminus\{1\}$ and $t\in \Omega_2\setminus\{1\}$. By \eqref{f7} and \eqref{f8} of Lemma \ref{R} and Remark \ref{r}, we get $G_{k,1}G_{k,1} G_{1,1}\equiv G_{1,1} G_{1,t} G_{1,t}$, $G_{1,j} G_{1,j} G_{1,1}  G_{1,1}\equiv  G_{1,1} G_{\ell,1} G_{\ell,1} G_{1,1}\equiv G_{1,1} G_{1,1} G_{1,t} G_{1,t}$ $(j\neq 1)$, $G_{1,1} G_{1, k-p}  G_{1, k-p}\equiv G_{1,1} G_{1, t}  G_{1, t}$ $(k-p\neq 1)$. Using this discussion in the last equation implies
\begin{align}\label{o2}
A_{1,i} A_{k,1} &\equiv \Delta_{k,\Omega_1}  \delta_{i,k} G_{1,1} G_{1,1} G_{1,t}G_{1,t}
+ \Delta_{k,\Omega_3} \big( \delta_{i,p+1}\delta_{k,p+1}  G_{1,1}  G_{1,1}G_{1,t}G_{1,t}  \\&\quad +\Delta_{i,\Omega_3}\widehat{\delta}_{i,p+1}\delta_{i,k}   G_{1,1} G_{1,1} G_{1, t}  G_{1, t}\big)\notag
\\& =  \delta_{ik} A_{1,1}. \notag
\end{align}
For the product $A_{i,1} A_{1,k}$: Using \eqref{ss1}(of the present proof) and \eqref{f1} of Lemma \ref{R}, we get
\begin{align*}
 A_{i,1} A_{1,k} &\equiv  \Delta_{i,\Omega_1}  G_{i,1} G_{1,1}A_{1,k}+ \Delta_{i,\Omega_3}(-  {A}_{1,i}A_{1,k} +  G_{1,i-p}A_{1,k})
 \\& \equiv \Delta_{i,\Omega_1} G_{i,1} G_{1,1} \big[\Delta_{k,\Omega_1}  G_{1,1} G_{k,1}  +  \delta_{k,p+1}  G_{1,t} G_{1,t} G_{1,1}
\\&\quad +  \Delta_{k,\Omega_3} \widehat{\delta}_{k,p+1}  G_{1,k-p}\big]
+ \Delta_{i,\Omega_3} G_{1,i-p}\big[\Delta_{k,\Omega_1}   G_{1,1} G_{k,1} \\ &\quad + \delta_{k,p+1}  G_{1,t} G_{1,t} G_{1,1}
 +  \widehat{\delta}_{k,p+1} \Delta_{k,\Omega_3}   G_{1,1} G_{1,1} G_{1,k-p}\big].
\end{align*}
Using \eqref{t2}, \eqref{t7}, \eqref{t3}, and \eqref{t1} of Corollary \ref{R2} implies
\begin{align*}
  & A_{i,1} A_{1,k}\equiv \Delta_{i,\Omega_1}\big[ \Delta_{k,\Omega_1} \left(- \delta_{i,k} G_{1,1} G_{1,1} G_{i,1} G_{i,1}+ G_{i,1}G_{k,1}\right) + \delta_{k,p+1} G_{i,1} G_{1,1} G_{1,1}
\\& \quad+ \Delta_{k,\Omega_3} \widehat{\delta}_{k,p+1} G_{i,1} G_{1,1}  G_{1,k-p}\big]
+ \Delta_{i,\Omega_3} \big[\Delta_{k,\Omega_1}  G_{1,i-p} G_{1,1} G_{k,1} \\&\quad+ \delta_{k,p+1}  \left(-\delta_{i,p+1} G_{1,1} G_{1,1} G_{1,t} G_{1,t}+ G_{1,i-p}G_{1,1}\right) +  \widehat{\delta}_{k,p+1} \Delta_{k,\Omega_3}  \big(\delta_{i,p+1}G_{1,1} G_{1,k-p}\\&\quad+ \widehat{\delta}_{i,p+1}   \left(-\delta_{i,k} G_{1,1} G_{1,1}G_{1,i-p} G_{1,i-p}+ G_{1,i-p} G_{1,k-p}\right)\big)\big]
\\
& \quad= A_{i,k}.
\end{align*}
Summarizing, for all $i,k \in (\Omega_1\cup \Omega_3)\setminus \{1\}$, we have
\begin{align}\label{pro2}
   A_{1,i} A_{1,k} = 0 = A_{i,1} A_{k, 1},\,\, A_{1,i} A_{k,1} = \delta_{ik} A_{1,1},\,\,  A_{i,1} A_{1,k} = A_{i,k}.
   \end{align}
We now use the products of  \eqref{pr1} and \eqref{pro2} to get all the others. For all $i,k, \ell , t\in (\Omega_1\cup \Omega_3)\setminus \{1\}$, we have
$A_{1,i}A_{i,1} = A_{1,1}$, hence  $A_{1,1} A_{1,i}A_{i,1} = A_{1,1} A_{1,1}$. Thus $A_{1,1} = A_{1,1}A_{1,1}$.  We also have $A_{i,k} = A_{i,1} A_{1,k}$. Hence $A_{i,k} A_{\ell, t} = A_{i,1} A_{1,k} A_{\ell, 1} A_{1, t}
= \delta_{k, \ell} A_{i,1} A_{1,1} A_{1,t}
 = \delta_{k,\ell} A_{i,1} A_{1,t}
  = \delta_{k,\ell} A_{i,t}$. Summarizing, for all $i,k, \ell , t\in \Omega_1\cup \Omega_3$, we have
  \[ A_{i,k} A_{\ell, t}  = \delta_{k,\ell} A_{i,t}.   \]
Now let $\mathcal{S}$ denote the subspace of  $\mathcal{U}(\mathcal{J}_{\field})$ generated by $A_{i,j}$ $(i,j \in \Omega_1\cup \Omega_3)$. We have shown that $\mathcal{S}$ is a subalgebra of $\mathcal{U}(\mathcal{J}_{\field})$ and is isomorphic to $M_{p+q \times p+q}(\field)$. By the definition of $A_{i,j}$, we get
\[ G_{i,j} = A_{i,j + p} + A_{j + p,i}\quad  (\text{for all}\,\,  i \in \Omega_1, \,\, j\in \Omega_2).\]
Thus all $G_{i,j}\in \mathcal{S}$. Hence $\mathcal{U}(\mathcal{J}_{\field}) \cong M_{p+q \times p+q}(\field)$.
\end{proof}

\begin{corollary}
The universal associative envelope of the Jordan triple system $\mathcal{J}_{\field}$ (of Definition \ref{def}) is semisimple.
\end{corollary}

\begin{corollary}\label{xo}
 The Jordan triple system $\mathcal{J}_{\field}$ (of Definition \ref{def})  has only one nontrivial representation (up to equivalence) defined by:
\[\Theta: \mathcal{J}_{\field} \to M_{p+q \times p+q}(\field), \quad E_{i,j} \to \left(
                                                               \begin{array}{cc}
                                                                O_{p\times p} & E_{i,j}  \\
                                                                 E_{j,i} & O_{q\times q}  \\
                                                               \end{array}
                                                             \right).\]
  \end{corollary}
\begin{proof}
By Theorem \ref{S}, the Jordan triple system $\mathcal{J}_{\field}$ has only one nontrivial representation. We now verify that $\Theta$ is a  representation of $\mathcal{J}_{\field}$. We first observe that,
\[\Theta(E_{i,j})= \mathbb{E}_{i, p + j} + \mathbb{E}_{p + j, i}\quad (\text{for all}\,\, i\in \Omega_1,\, j\in \Omega_2).\]
where $\mathbb{E}_{i, j}$  is the $(p+q) \times (p+q)$ matrices whose $(i,j)$ entry is $1$ and all the other entries are 0.
For all $i, k, s\in \Omega_1$ and $j,\ell, t \in \Omega_2$, we have
 \begin{align*}
\Theta\{E_{i,j}, E_{k,\ell}, E_{s,t}\} & = \Theta\left(E_{i,j} E_{\ell,k} E_{s,t}+ E_{s,t} E_{\ell,k} E_{i,j}\right)
 \\& = \Theta\left(\delta_{j,\ell}\delta_{k,s} E_{i,t}+ \delta_{t,\ell} \delta_{k,i} E_{s,j}\right)
 \\
 &= \delta_{j,\ell}\delta_{k,s} \left(\mathbb{E}_{i,p+t} + \mathbb{E}_{p+t,i}\right) + \delta_{t,\ell} \delta_{k,i} \left(\mathbb{E}_{s, p+j}+ \mathbb{E}_{p+j,s}\right).
 \end{align*}
 On the other hand
 \begin{align*}
 \{\Theta (E_{i,j}), \Theta (E_{k,\ell}), \Theta (E_{s,t})\}& =  \Theta (E_{i,j})(\Theta (E_{k,\ell}))^t\Theta (E_{s,t})+\Theta (E_{s,t})(\Theta (E_{k,\ell}))^t\Theta (E_{i,j})
\\
& = \left(\mathbb{E}_{i,p+j} + \mathbb{E}_{p+j,i}\right)\left(\mathbb{E}_{k,p+\ell}+ \mathbb{E}_{p+\ell,k}\right)^t \left(\mathbb{E}_{s,p+t}+ \mathbb{E}_{p+t,s}\right) \\&\quad+ \left(\mathbb{E}_{s,p+t} + \mathbb{E}_{p+t,s}\right) \left(\mathbb{E}_{k, p+\ell} + \mathbb{E}_{p+\ell,k})^t (\mathbb{E}_{i,p+j}+ \mathbb{E}_{p+j,i}\right)
\\
&= \left(\mathbb{E}_{i,p+j} + \mathbb{E}_{p+j,i}\right)\left(\mathbb{E}_{k,p+\ell}+ \mathbb{E}_{p+\ell,k}\right) \left(\mathbb{E}_{s,p+t}+ \mathbb{E}_{p+t,s}\right) \\&\quad+ \left(\mathbb{E}_{s,p+t} + \mathbb{E}_{p+t,s}\right) \left(\mathbb{E}_{k, p+\ell} + \mathbb{E}_{p+\ell,k}\right) \left(\mathbb{E}_{i,p+j}+ \mathbb{E}_{p+j,i}\right)
\\
& = \delta_{j,\ell}\delta_{k,s} \mathbb{E}_{i,p+t}\!\!+\!\! \delta_{i,k} \delta_{\ell,t} \mathbb{E}_{p+j,s}\!\!+\!\! \delta_{t,\ell} \delta_{i,k} \mathbb{E}_{s,p+j}\!\! +\!\! \delta_{s,k} \delta_{\ell,j}\mathbb{E}_{p+t,i}.
  \end{align*}
Hence $\Theta$ is a representation of $\mathcal{J}_{\field}$.
\end{proof}

\begin{lemma}\label{center}
The center $\mathfrak{C}(\mathcal{U}(\mathcal{J}_{\field}))$ of the universal associative envelope $\mathcal{U}(\mathcal{J}_{\field})$ has dimension $1$ with a basis
\[ e = (1-q)\,G_{1,1} G_{1,j}G_{1,j} \! +\! (1-p)\, G_{1,1}G_{1,1} G_{t,1}G_{t,1} \!+\! \sum^p_{i=2}\! G_{i,1} G_{i,1} \!+\! \sum^q_{s = 1}\! G_{1,s} G_{1,s},\]
for any $t \in \Omega_1 \setminus\{1\} $ and $j \in \Omega_2\setminus\{1\}$.
\end{lemma}
\begin{proof}
By Theorem \ref{S}, we have $\mathfrak{C}(\mathcal{U}(\mathcal{J}_{\field})) \cong \field$. It follows that $e= \sum^{p+q}_{i=1} A_{i,i}$ is the only idempotent in $\mathcal{U}(\mathcal{J}_{\field})$ that span the center. Using the proof of Theorem \ref{S}, we get
\begin{align*}
e &= G_{1,1} G_{1,1} G_{1,j} G_{1,j} + \sum^p_{i=2} \left(- G_{11} G_{1,1} G_{t,1} G_{t,1}+ G_{i,1} G_{i,1}\right) \\&\quad + \sum^{p+q}_{i= p+1} \left(- G_{1,1}G_{1,1} G_{1,j} G_{1,j} + G_{1,i-p} G_{1,i-p}\right)
\\
& =  G_{1,1} G_{1,1} G_{1,j} G_{1,j} -(p-1)\,  G_{11} G_{1,1} G_{t,1} G_{t,1}+ \sum^p_{i=2}  G_{i,1} G_{i,1} \\&\quad -q\, G_{1,1}G_{1,1} G_{1,j} G_{1,j} + \sum^{p+q}_{i= p+1} G_{1,i-p} G_{1,i-p}.
\end{align*}
This completes the proof.
\end{proof}
\section*{Acknowledgments}
The author thanks the anonymous referee for helpful comments.

\end{document}